\documentclass[conference]{IEEEtran}
\IEEEoverridecommandlockouts
\usepackage{cite}
\usepackage{amsmath,amssymb,amsfonts,amsthm,epsfig,nccmath}
\usepackage{algorithm}
\usepackage{algpseudocode} % For pseudocode formatting
\usepackage{float}
\usepackage{graphicx}\usepackage{amsmath}
\usepackage{amssymb}
\usepackage{hyperref}
\usepackage{textcomp,subcaption}
\usepackage{bbm}
\usepackage{tikz, verbatim}
\usepackage{subfiles}
\usepackage{multirow}
\usepackage{xcolor}
\usepackage[bottom,norule,hang]{footmisc}
\usepackage{array}
\usepackage{enumitem}
\usepackage{geometry}
\usepackage{resizegather}
\newtheorem{theorem}{Theorem}
\newtheorem{lemma}{Lemma}
\newtheorem{remark}{Remark}

\newtheorem{assumption}{A}
\DeclareMathOperator{\sign}{sign}

\newcommand{\remove}[1]{}

\newcommand{\ALG}{$\mathcal{A}$ }
\newcommand{\ALGone}{A-OBD }
\newboolean{showcomments}
\setboolean{showcomments}{true}
\newcommand{\hs}[1]{\ifthenelse{\boolean{showcomments}} 
	{\textcolor{magenta}{(HS says: #1)}} {} }
\newcommand{\pc}[1]{ \ifthenelse{\boolean{showcomments}}
	{\textcolor{red}{(PC says: #1)}} {} }
\newcommand{\rz}[1]{\ifthenelse{\boolean{showcomments}} 
	{\textcolor{blue}{(RZ says: #1)}} {} }
\newenvironment{customproof}[1][]
{\par\noindent\textit{Proof of #1. }\ignorespaces}
{\qed\par\vspace{1em}}

\usepackage{amsfonts}

\def\bb0{{\mathbb{0}}}

\def\bb{{\mathbf{b}}}

\def\b0{{\mathbf{0}}}

\def\opt{\mathsf{OPT}}

\def\b1{{\mathbf{1}}}

\def\bbR{{\mathbb{R}}}

\def\cA{\mathcal{A}}

\def\cO{\mathcal{O}}

\def\sfd{{\mathsf{d}}}

\def\sf0{{\mathsf{0}}}

\begin{document}
	\title{Online Convex Optimization with Switching Cost with Only One Single  Gradient Evaluation} 
    \author{
\IEEEauthorblockN{Harsh Shah\IEEEauthorrefmark{1},
Purna Chandrasekhar\IEEEauthorrefmark{2},
Rahul Vaze\IEEEauthorrefmark{3}}
\IEEEauthorblockA{\IEEEauthorrefmark{1}Indian Institute of Technology Bombay (IITB)\\
Email: 210100063@iitb.ac.in}
\IEEEauthorblockA{\IEEEauthorrefmark{2}Indian Institute of Technology Delhi (IITD)\\
Email: mt6210270@iitd.ac.in}
\IEEEauthorblockA{\IEEEauthorrefmark{3}Tata Institute of Fundamental Research (TIFR)\\
Email: rahul.vaze@gmail.com}
}

	\maketitle
	
	\begin{abstract}
		Online convex optimization with switching cost is considered  under the frugal information setting where at time $t$, before action $x_t$ is taken, 
		only a single function evaluation and a single gradient is available at the previously chosen action $x_{t-1}$ for either the current cost function $f_t$ or the most recent cost function $f_{t-1}$. When the switching cost is linear, online algorithms with optimal order-wise competitive ratios are derived for the frugal setting. When the gradient information is noisy, an online algorithm whose competitive ratio grows quadratically with the noise magnitude is derived.
	\end{abstract}

	\section{Introduction}
Online convex optimization with switching cost (OCO-S) problem was introduced in \cite{lin2012online}, 
where for each discrete time $t$, a convex cost function is $f_t: \bbR^{\sfd} \rightarrow \bbR^+$. The cost of choosing action $x_t$ at time $t$ is the sum of the cost $f_t(x_t)$, and the switching cost
$c(x_{t-1}, x_t)$. The goal is to choose $x_t, 1\le t\le T$ that solves
\begin{equation}\label{eq:total-cost}
\min_{x(t)} \sum_{t=1}^T f_t(x_t) + c(x_{t-1}, x_t).
\end{equation}

Problem \eqref{eq:total-cost} was motivated by capacity provisioning problem  \cite{lin2012online}
where demand $d(t)$ at time $t$ is revealed causally over time, and to serve demand $d(t)$, $s(t)$ number of servers are employed.  Changing $s(t)$ over time to dynamically serve $d(t)$ keeps the QoS cost $f(d(t), s(t))$ small but incurs a large switching cost. 
%With this motivation, Problem \eqref{eq:total-cost} was formulated.

\subsection{Prior Work}
For the OCO-S \eqref{eq:total-cost} with quadratic switching cost, i.e., $c(x_{t},
x_{t - 1}) = ||x_{t} - x_{t - 1}|| ^ 2$, if the functions \(f_{t}\)'s
are just convex, then the competitive ratio of any
deterministic online algorithm is unbounded
\cite{goel2019beyond}. Thus, starting with \cite{chen2018}, to solve
\eqref{eq:total-cost}, $f_t$'s were assumed to be strongly convex. 

With
\(\mu\)-strongly convex $f_t$'s, online balanced descent (OBD)
algorithm that balances the suboptimality and switching cost at each
time, has a competitive ratio of at most \(3 +
\mathcal{O}(\frac{1}{\mu})\) \cite{goel2019online}. Surprisingly, \cite{ZhangSOCO2021XstarAlgorithm} showed that the same guarantee can be 
obtained by a simple algorithm that chooses $x_t$ as the local optimizer of $f_t$ disregarding the switching cost. This result was
improved in \cite{goel2019beyond} using a modification of OBD to get a
competitive ratio $\mathcal{O}(\frac{1}{\sqrt{\mu}})$ as \({\mu \to
  0^{+}}\), which is also shown to be the best possible (order-wise)
competitive ratio for all possible deterministic online algorithms
\cite{goel2019beyond}.

%Similar results have been obtained 
%when $f_t$'s are locally polyhedral in \cite{chen2018}. 
%Very recently, \cite{bhuyan2023best} studied OCO-S under some assumptions on the sequence of minimizers of $f_t$'s and provided an algorithm that works for both the worst case as well as the stochastic input. The OCO-S has also been studied in the information setting of OCO recently \cite{senapativaze2023}, where action $x_t$ is chosen before $f_t$ is revealed.

OCO-S with linear switching cost, where $c(x_{t}, x_{t - 1})= |x_{t}
- x_{t - 1}|$ has also been widely considered. For this case, when $\sfd=1$, an
optimal offline algorithm was derived in \cite{lin2012online}, while
online algorithms have been explored in \cite{lin2012online,
  lin2012dynamic, andrew2013tale, bansal20152,
  DBLP:journals/corr/abs-1807-05112,chen2018, argue2020dimension, chen2018smoothed} when
$f_t$'s are just convex.  For $\sfd=1$, the best known online algorithm
has the competitive ratio of $2$ which matches the lower bound
\cite{DBLP:journals/corr/abs-1807-05112}. 
%It is important to notice
%that unlike the quadratic switching cost case, when the switching cost
%is linear, the $f_t$'s are not required to be strongly convex to
%get meaningful competitive ratios.  
For general $\sfd$, when \(f_{t}\)'s
are convex, the competitive ratio of any algorithm has a lower bound
\(\Omega(\sqrt{\sfd})\) \cite{chen2018}. 
 \cite{argue2020dimension} proposed a
constrained online balanced descent (COBD) algorithm with a competitive
ratio \(\cO(\sqrt{\kappa})\) where \(\kappa := L /
\mu\) as the condition number. Recently, \cite{soco-constraints} considered the OCO-S with linear switching cost under some long-term constraints.

Almost all the prior work except \cite{senapativaze2023} considered the OCO-S problem under the full-information setting
where the action $x_t$ is chosen after function $f_t$ is  revealed completely.
The full information setting is quite imposing for many applications 
and difficult to realize in practice. In \cite{senapativaze2023}, when $f_t$'s are strongly convex, the limited information setting was considered, where at time $t$, multiple stale gradients $\nabla f_{t-1}(x_1), \dots, \nabla f_{t-1}(x_n)$ are available, where $n$ is a function of the strong convexity parameter and Lipschitz constant.

In this paper, we study the  OCO-S problem for $\sfd=1$ when functions $f_t$'s are convex under the frugal information setting, where at time $t$, before action $x_t$ is chosen, either $\{f_{t}(x_{t-1}), \nabla f_{t}(x_{t-1})\}$ (called {\bf fresh}) or $\{f_{t-1}(x_{t-1}), \nabla f_{t-1}(x_{t-1})\}$ (called {\bf stale}) is available. Thus, information available is only a single function evaluation and a single gradient at the previously chosen action $x_{t-1}$ for either $f_t$ or $f_{t-1}$.
We also consider the noisy gradient setting. We limit ourselves to the linear switching cost, i.e.,  $c(x_{t}, x_{t - 1})= |x_{t}
- x_{t - 1}|$.

In this paper, we consider the worst-case input and the objective is to design online (that have access to only causal information) algorithms for solving Problem \eqref{eq:total-cost} with small
competitive ratios that is defined as follows. 
The competitive ratio for an online algorithm $\cA$ is
$\mu_\cA = \max_{\sigma}\frac{C_\cA(\sigma)}{C_\opt(\sigma)},$
where $\sigma$ is the input, $C_\cA$  is the cost  \eqref{eq:total-cost} of an online algorithm $\cA$, while \(C_\opt\) denotes the cost
of the optimal offline algorithm $\opt$, that knows the input $\sigma$
ahead of time, non-causally.

\subsection{Our Contributions with linear switching cost }
We propose a single online algorithm (A-OBD) for all the frugal settings, where only the input parameter changes depending on the particular setting.
\begin{itemize}
\item For the fresh setting, we show that the $\mu_{\text{A-OBD}}= O(M)$, where $M$ is the common smoothness parameter of functions $f_t$'s. We also show that the competitive ratio of any online algorithm is $\Omega(M)$.
\item For the stale setting, we show that the $\mu_{\text{A-OBD}}= \Theta (M+L)$, where $L$ is the common Lipschitz constant for functions $f_t$'s. We also show that the competitive ratio of any online algorithm is $\Omega(\max\{M, L\})$.
\item With noisy gradients, we show that in the fresh/stale setting, $\mu_{\text{A-OBD}}\le O(M+L+\alpha^2)$, where $\alpha$ is maximum noise error defined ahead. 
%The stale counterpart of this result has an additional additive $L$ term in the competitive ratio bound.
\end{itemize}

	\section{System Model}
%	We refer to Problem \ref{eq:total-cost} as the \textbf{OCO with Linear Switching Cost (OCO-LS) Problem using only Gradient information:} Let the cost function at time $t=1,2,\ldots,T$, be a non-negative convex function $f_t:[0,1]\rightarrow\mathcal{R^+}$. The objective is to choose action $x_t$ at time $t$ so as to minimize the cost 
%	%	Objective of an online algorithm \ALG is to choose an action $x_t \in [0,1]$ so as to minimize the total cost.\\
%	\begin{equation}
%		C_{\mathcal{A}}=\sum_{t=1}^{T}f_t(x_t)+\sum_{t=1}^{T}c(x_t,x_{t-1})
%		\label{eq:total-cost}
%	\end{equation}
%	where $c(x_t,x_{t-1})$ represents the cost incurred in changing decisions across time steps. 
%	
We consider Problem \eqref{eq:total-cost} with $\sfd=1$.
	Typically, $x_t$ is chosen to minimize \eqref{eq:total-cost} with the full information about $f_t$ which can be expensive to obtain. To obviate this, in this paper, we consider frugal information settings. For example, at time $t,$ either $\{f_t(x_{t-1}), \nabla f_t(x_{t-1})\}$ or  $\{f_{t-1}(x_{t-1}), \nabla f_{t-1}(x_{t-1})\}$ or an approximate estimate of $\nabla f_{t-1}(x_{t-1})$ is available. Under these frugal information settings, we consider the design of online algorithms that minimize the cost \eqref{eq:total-cost} using only causal information. For the rest of this paper, we consider \eqref{eq:total-cost} with linear switching cost, i.e., $c(x_t,x_{t-1})=|x_t-x_{t-1}|,$.

	%	The total cost $C_{\mathcal{A}}$ comprises of the hitting cost $f_t(x_t)$ and the switching cost $c(x_t,x_{t-1}).$ We consider linear switching cost in this paper, i.e. $c(x_t,x_{t-1})=|x_t-x_{t-1}|.$ The goal for \ALG is to minimize the competitive ratio $\mu_{\mathcal{A}}$ given by
	%The metric we consider to evaluate the performance of any online algorithm \ALG is the competitive ratio $\mu$ defined as follows \vspace{-0.1in}
%	\begin{equation}
%		\mu_{\mathcal{A}} := \max_{f_t} \frac{C_{\mathcal{A}}}{C_{OPT}},
%		\label{eq:CR}
%	\end{equation}	
%	where $C_{OPT}$ denotes the cost of the optimal algorithm (OPT) that knows all the $f_t$'s ahead of time.

	The precise definition of causal frugal information settings that we consider is as follows
	\begin{enumerate}[label=(\alph*)]
		\item \textbf{fresh gradient setting:} At time $t,$ to choose $x_t,$ \ALG can use $\left\{ f_i(x_{i-1}), \nabla f_i(x_{i-1}) \mid i = 1, 2, \ldots, t \right\}.$ \label{setting:fresh-grad}
		\item \textbf{stale gradient setting:}  At time $t,$ to choose $x_t,$ \ALG can use  $\left\{ f_{i-1}(x_{i-1}), \nabla f_{i-1}(x_{i-1}) \mid i = 2, \ldots, t \right\}.$ \label{setting:stale-grad}
		\item \textbf{noisy gradient setting:} Let $\nabla f^{\text{rec}}_t$ be the approximate estimate of $\nabla f_t(x_{t-1})$ with error at most $\alpha$ i.e. $|\nabla f^{\text{rec}}_t-\nabla f_t(x_{t-1})| \leq \alpha.$
		At time $t,$ to choose $x_t,$ \ALG can use $\left\{ f_i(x_{i-1}), \nabla f^{\text{rec}}_i \mid i = 1, 2, \ldots, \tau \right\}.$ We consider this noisy setting for both  fresh ($\tau=t$) and stale ($\tau=t-1$) scenarios.
		\label{setting:noisy-grad}
	\end{enumerate}
	\section{Algorithm}
	At time $t$, let $t_{\max}$ be the latest time $t'\le t$ for which information about $f_{t'}(x_{t-1})$ and $\nabla f_{t'}(x_{t-1})$ is available. Thus, in the fresh setting $t_{\max}=t$, while for the stale setting $t_{\max}=t-1$.
	%Let $\hat{f}_t$ be the available function at the time step $t$ before taking action $x_t.$ For the setting \ref{setting:fresh-grad}, $\hat{f}_t$ corresponds to $f_t$ itself, while for the setting \ref{setting:stale-grad}, $\hat{f}_t$ corresponds to $f_{t-1}.$ 
%	Define $f^{avl}_t:={f}_{t_{\max}}(.)$ and $\nabla^{avl}_t:=\nabla {f}_{t_{\max}}(.).$ In the noisy gradient setting, $f^{avl}_t$ remains same as above but $\nabla^{avl}_t$ is equal to $\nabla {f}^{\text{rec}}_{t_{\max}}(.)$ due to error in the received gradient information.
	Define $f^{avl}_t:={f}_{t_{\max}}(x_{t-1})$ and $\nabla^{avl}_t:=\nabla {f}_{t_{\max}}(x_{t-1}).$ In the noisy gradient setting, $f^{avl}_t$ remains same as above but $\nabla^{avl}_t$ is equal to $\nabla {f}^{\text{rec}}_{t}$ due to error in the received gradient information.
	
	%\hs{$f_t^{avl}$ is a single value of the function ${f}_{t_{\max}}(.)$ at $x_{t-1}.$}
	%[XXX: avl stands for available information, you had val for some reason]
	%\hs{In the noisy case $\nabla f_t$ is the true value but the one which we receive from oracle is different which is defined as $\nabla f_t^{rec}$. The $f^{avl}$ defined in the ALG is for denoting the function which has come the latest and what information values it has given us. I am using rec in the noisy case so for the received information.}
	
	We now describe the algorithm \textbf{Approximate Online Balanced Descent algorithm} (\textbf{A-OBD}), that will be used to choose action $x_t$ for all the frugal information settings  with only difference being the choice of the parameter $\delta_{t}$, which varies depending on the specific setting. The main idea behind {A-OBD} is as follows.
	
	Define two lines $L_1(x):$ $f^{avl}_t + \nabla^{avl}_t \cdot (x - x_{t-1})$ i.e., the gradient line $\nabla^{avl}_t$ at $x_{t-1}$, and $ L_2(x):$ line with slope \( \sign(-\nabla_t^{val})\delta_t \) passing through $x_{t-1}$ on the $x-$axis. 
	
	At every time step $t,$ {A-OBD} chooses $x_t$ as the intersection point of $L_1(x)$ and $L_2(x).$ Since $L_1$ is the first order approximation of $f^{avl}_t$ (given the information) the motivation behind this choice is to approximately make the hitting cost $f^{avl}_t(x_t)$ equal to $\delta_{t}$ times the switching cost $|x_t-x_{t-1}|.$ This allows the comparison of cost of {A-OBD} with that of the OPT to upper bound the competitive ratio of {A-OBD}.
	
	Define the hitting cost at time \( t \) as \( \mathcal{H}_t := f_t(x_t) \) and the switching cost as \( \mathcal{M}_t := |x_t - x_{t-1}| \). An exact balanced point \( \hat{x}_t \) is defined as the point where 
	\begin{equation}\label{eq:balancedpoint}
		\mathcal{H}_t = \delta_t \mathcal{M}_t.
	\end{equation}
	Since \ALGone operates in a frugal information setting, it chooses $x_t$ that is an approximation of \( \hat{x}_t \) rather than the exact balanced point.

	%	which corresponds to the first order approximation of the point where hitting cost is $\delta_t$ times the switching cost. Essentially, \ALGone updates its position by approximately balancing the trade-off between the hitting cost and the movement cost. The same algorithm is used in all three information settings with different $\delta_t$ values.	\hs{how do I write about relation of $\delta_{t}$ with different info settings?}
	\begin{algorithm}  
		\caption{A-OBD($\delta_t$)} \label{alg:AOBD}
		\begin{algorithmic}[1]
			\For{$t = 1, \ldots, T$}
			\State $L_1(x) \gets f^{avl}_t + \nabla^{avl}_t \cdot (x - x_{t-1})$ 
			\State $L_2(x) \gets \delta_t \cdot \sign(-\nabla^{avl}_t) \cdot (x - x_{t-1})$ 
			\State Find $x_t$ such that $L_1(x) = L_2(x)$ 
			\State \Return $x_t$ 
			\EndFor
		\end{algorithmic}
	\end{algorithm}	
	%	\begin{algorithm}
		%		\caption {A-OBD}
		%		\begin{algorithmic}[1]
			%			\State\textbf{Input:} $x_{0} (\text{starting point)}, $;
			%			\State Choose \(x_{1} = x_0\);
			%			\State\textbf{for} \(t = 2, \ldots, T\)
			%			\State\hspace{3mm} Set \(z_{t} ^ {(0)} = x_{t - 1}\);
			%			\State\hspace{3mm} \textbf{for} \(k = 1, \ldots, K\)
			%			\State\hspace{3mm}\hspace{3mm} Query \(\nabla f_{t - 1}(z_{t} ^ {(k - 1)})\);
			%			\State\hspace{3mm}\hspace{3mm} Update \(z_{t} ^ {(k)} = px{\mathcal{X}}{z_{t} ^ {(k - 1)} - \frac{1}{L}\nabla f_{t - 1}(z_{t} ^ {(k - 1)})}\);
			%			\State\hspace{3mm}\textbf{end}
			%			\State\hspace{3mm} Choose \(x_{t} = z_{t} ^ {(K)}\);
			%			\State\textbf{end}
			%		\end{algorithmic}
		%		\label{alg:omgd}
		%	\end{algorithm}
			\section{Fresh Gradients Setting  }\label{sec:fresh-grad}
	%In this section, we consider OCO-LS problem in setting \ref{setting:fresh-grad} and provide upper bound on $\mu_{\text{A-OBD}}$ and also derive lower bound on $\mu_{\mathcal{A}}$ for any online algorithm $\mathcal{A}$.
	
	Define $x_t^*:= \arg \min_{x \in [0,1]} f_t(x).$ We make the following assumption in this setting.
	\begin{assumption}\label{asm:a1}
		For all $1 \leq t \leq T$, the function $f_t:[0,1]\rightarrow\mathcal{R^+}$ is $M$-smooth, convex and $f_t(x_t^*)=0$\footnote{This is without loss of generality.}. 
	\end{assumption}

	%	Here the A-OBD{$(f_{rec}, \nabla_t^{avl}, x_{t-1}, \delta_t)$} uses constant $\delta_t=\delta_1 \ \forall \ t$ as given in Theorem~\ref{Theorem:UB-fresh}.
	\begin{theorem} \label{Theorem:UB-fresh}
		Under the assumption A\ref{asm:a1}, in the fresh gradient setting \ref{setting:fresh-grad}, the competitive ratio of A-OBD{$(\hat{\delta})$} satisfies
		\begin{align}
			\mu_{\text{A-OBD}\left(\hat{\delta}\right)} \leq \frac{M}{4}+\frac{3}{2}+\frac{1}{2}\sqrt{\left(\frac{M}{2}+1\right)^2+7} = O(M)
		\end{align}
		when the parameter $\hat{\delta} = \frac{1}{2} \sqrt{\left(\frac{M}{2} + 1\right)^2 + 7} - \frac{M}{4} + \frac{1}{2}.$
	\end{theorem}
	The competitive ratio of A-OBD scales linearly with $M$. 
	The dependence on the smoothness parameter \( M \) in Theorem~\ref{Theorem:UB-fresh} arises because of the discrepancy between the approximated balanced point and the exact balanced point 
	$ \hat{x}_t$ \eqref{eq:balancedpoint} increases with \( M \). As \( M \) grows, the function $f_t$'s curvature becomes more pronounced, making the first-order approximation less accurate and causing greater deviation from the exact balanced point.  
	
	%This is really on account of only knowing the gradient of $f_t$ and not full function $f_t$, since two functions with the same gradient and smoothness parameter $M$ can be sufficiently far apart and A-OBD
	
	%Theorem~\ref{Theorem:UB-fresh} establishes an upper bound for A-OBD, showing that its growth is at most of order \( M \) when using a constant \(\delta_t\) for all \( t \).  

	The proof of Theorem~\ref{Theorem:UB-fresh} follows the potential function method used in \cite{bansal2015}. However, unlike \cite{bansal2015}, where $x_t$ is  \( \hat{x}_t \) \eqref{eq:balancedpoint}, A-OBD(\(\hat{\delta}\)) chooses $x_t$ such that \( \mathcal{H}_t \) is bounded from both above and below by linear multiples of \( \mathcal{M}_t \), ensuring an approximate balance between the hitting and the movement cost. While \cite{bansal2015} uses a fixed $\delta_t=2,$ here $\hat{\delta}$ depends on $M$ because A-OBD moves to an approximated balanced point rather than the exact one. A detailed proof is provided in Appendix~\ref{proof:theorem-1}.

	Next, we show that the dependence of $M$ on the competitive ratio of any online algorithm is fundamental.
	
	\begin{theorem}\label{Theorem:LB-fresh}
		Under the assumption A\ref{asm:a1}, in the fresh gradient setting \ref{setting:fresh-grad}, any online algorithm \ALG must have a competitive ratio of at least $\Omega(M).$ 
	\end{theorem}
	
	We obtain Theorem~\ref{Theorem:LB-fresh} by constructing functions that appear identical to any online algorithm but are inherently different. 
	%	Since \ALG only receives function and gradient information at \( x_{t-1} \), if two distinct functions share the same values and gradients at these points, \ALG cannot distinguish between them.  
	The basic idea of the proof is as follows. Consider two \( M \)-smooth functions, \( g_1 \) and \( g_2 \), satisfying \( g_1(x_0) = g_2(x_0) \) and \( g_1'(x_0) = g_2'(x_0) \). Since we are in the fresh gradient setting,  any online algorithm cannot distinguish between $g_1$ and $g_2$, and will choose the same action $x_1$ at time $1$. 
	However, the action $x_1^{OPT}$ chosen by the OPT will vary depending on the structural differences between the full functions \( g_1 \) and \( g_2 \). To derive the lower bound, we will choose $g_1$ and $g_2$ in such a way that $x_1^{OPT}$ is significantly different when the functions are \( g_1 \) and \( g_2 \). A detailed proof is presented in appendix \ref{proof:Theorem-2}.
	
	% [XXX: Move everything after this to the appendix]
	
	\section{Stale Gradient setting}\label{sec:stale-grad}
	
	%	In this setting, action $x_t$ is taken at time $t$ by using the received function value and the gradient value of the previous function $f_{t-1}$ at the previously moved action $x_{t-1}.$
	In this section, we consider OCO-S problem in setting \ref{setting:stale-grad} and provide upper bound on $\mu_{\text{A-OBD}}$. For studying the stale gradient setting \ref{setting:stale-grad}, in addition to  assumption A\ref{asm:a1} we will make an additional assumption: 
	\begin{assumption}\label{asm:a2}
		For all $1 \leq t \leq T$, the function $f_t:[0,1]\rightarrow\mathcal{R^+}$ is $L$-Lipschitz.
	\end{assumption}

	%Here the \ALGone $(\delta_{t})$ uses constant $\delta_t=\delta_2 \ \forall \ t$ as specified in Theorem~\ref{Theorem:UB-stale}.

	\begin{theorem}\label{Theorem:UB-stale}
		Under the assumption A\ref{asm:a1}. and A\ref{asm:a2}. in the stale gradient setting \ref{setting:stale-grad}, the competitive ratio of A-OBD($ \tilde{\delta}$) satisfies
		
		\begin{equation}
			\begin{gathered}
				\begin{aligned}
					\mu_{\text{A-OBD} \ (\tilde{\delta})} & \leq \frac{M+2L+6}{4}+ \nonumber \frac{\sqrt{(M+2L-2)^2-8(M+4L+4)}}{4}, \nonumber \\
					& = \mathcal{O}(M+L)
				\end{aligned}
			\end{gathered}
		\end{equation}
		when the parameter 
		\begin{equation} \nonumber
			\begin{gathered}
				\begin{aligned}
					\tilde{\delta} =  \frac{\sqrt{(M+2L-2)^2-8(M+4L+4)}}{4}  - \frac{M+2L-2}{4} .
				\end{aligned}
			\end{gathered}
		\end{equation}
	\end{theorem}
	%	Theorem~\ref{Theorem:UB-stale} establishes an upper bound for A-OBD, showing that its growth is at most of order \( M+L \) when using a constant \( \delta_t \) for all \( t \).  
	For a fixed input sequence ${\cal F} = \{f_t\}_{t=1}^T$, let the actions of  algorithm A-OBD($\tilde{\delta}$) in the fresh gradient setting be  $x_t^{fresh}$. Then actions of  algorithm A-OBD($\tilde{\delta}$) for ${\cal F}$ in the stale gradient setting are $x_t^{stale} = x_{t-1}^{fresh}$. This allows us to relate the cost  of A-OBD($\tilde{\delta}$) in the fresh and stale gradient setting, and we derive 
	Theorem \ref{Theorem:UB-stale} using Theorem \ref{Theorem:UB-fresh}'s result while replacing $\hat{\delta}$ with $\tilde{\delta}$. 
%	[XXX: Use this explanation in your proof too, currently its not clear why certain facts are true]
	
	%	let $C^{fresh} = \sum_{t=1}^{T}f_t(x_{t}) + \sum_{t=1}^{T}|x_{t}-x_{t-1}|$ while
	%					$C^{stale}  = \sum_{t=1}^{T}f_t(x_{t-1}) + \sum_{t=1}^{T-1}|x_{t}-x_{t-1}|$, where the only difference is the action at which $f_t$ is evaluated, which is $x_t$ in $C_{\mathcal{A}}^{fresh}$ and $x_{t-1}$ in $C_{\mathcal{A}}^{stale}$.

	%	Define \( \mathcal{C}^{\text{fresh}}_{\mathcal{A}} \) and \( \mathcal{C}^{\text{stale}}_{\mathcal{A}} \) as the total cost of algorithm \( \mathcal{A} \) in settings \ref{setting:fresh-grad} and \ref{setting:stale-grad}, respectively. When \( \mathcal{A} \) is applied in both settings with the same loss functions \( f_t \) for all \( t \), the actions taken in the stale gradient setting are same as fresh gradient setting but delayed by one time step. Consequently, the total switching costs remain the same in both settings, while the hitting cost differs since actions are evaluated on different functions.	
	
	%Using this relationship, Lemma~\ref{lemma:cost-relation} upper bounds $\mathcal{C}^{\text{stale}}_{\mathcal{A}}$ in terms of    $\mathcal{C}^{\text{fresh}}_{\mathcal{A}}$ and total switching cost. Using the previous lemma, we can transform the problem of bounding $\mathcal{C}^{\text{stale}}_{\mathcal{A}}$ to bounding $ \mathcal{C}^{\text{fresh}}_{\mathcal{A}}$ with a higher switching cost.
	
	Note that the upper bound for \ALGone in stale gradient setting grows at most with order $M+L$. The $M$ contribution is similar to the fresh gradient setting, while the Lipschitz constant $L$ is on account of the delay from stale gradients. 	
	% Denote the costs incurred by the algorithm in the fresh and stale gradient settings as $C_{\text{A-OBD} \ (\delta_{2})}^{fresh}$ and $C_{\text{A-OBD} \ (\delta_{2})}^{stale}$, respectively. Note that we are using same $\delta$ parameter in both settings which differs from the parameter used in the fresh gradient setting discussed in Section~\ref{sec:fresh-grad}. 
	
	%        This difference can be bounded in terms of $\mathcal{M}_t$ using the Lipschitz property of the functions.
	
	% Obtaining this relationship, we observe that $C_{\text{A-OBD} \  (\delta_{2})}^{stale}$ is bounded above by a cost similar to $C_{\text{A-OBD} \  (\delta_{2})}^{fresh}$, but with a higher movement cost. Hence the competitive ratio can be upper bounded in a similar way. This result is formalized as follows: 

	\subsection{Lower Bound on Competitive ratio}
	Let $x_0=0$ without loss of generality. Then let $f_1(x)$ be either $x-1$ or $L(x-1)$ and $T=1$. Clearly $C_{OPT}$ is independent of $L$, while $C_{{\cal A}} = \Omega(L)$. Thus, we get that in the stale gradient setting, $\mu_{{\cal A}}= \Omega(L)$ and combining Theorem \ref{Theorem:LB-fresh}, we get $\mu_{{\cal A}}= \Omega(\max\{L,  M\})$.
	%[XXX: Just appeal to Spandan's result for $\Omega(L)$ even for multiple delayed gradients and combine it with Theorem 2 to get $\Omega(M,L)$ as the lower bound.
	%\hs{This part is left to do.}
	\section{Noisy Gradient setting}	
	We only present the case of fresh noisy gradient, where the result for stale case follows similar to Section \ref{sec:stale-grad}.
	%In this setting, action $x_t^{rec}$ is taken at time $t$ by using the gradient value of the present function $f_t$ at the previously moved action $x_{t-1}^{rec}$. However, the gradient information provided to the algorithm is corrupted by additive noise, introducing errors into the decision process. At each time step $t$, the received gradient $\nabla_{t}^{rec}$ deviates from the true gradient $\nabla_{t}^{true}$ by at most a tolerance of $\alpha >0$, satisfying:
	Recall that $\nabla^{avl}_t$ is the approximate estimate  of $\nabla_t(x_{t-1})$ available to any online algorithm
	with error at most $\alpha$ (i.e. $|\nabla^{avl}_t-\nabla_t | \leq \alpha$), where $\nabla_t:=\nabla f_t(x_{t-1})$.

	Note that $\alpha$ remains unknown and is not used by algorithm A-OBD, and  the guarantees we derive are a function of the unknown $\alpha$.
	In addition to  A\ref{asm:a1}, A\ref{asm:a2} we need to make two additional assumptions:
	
	\begin{assumption}\label{asm:a3}
		The magnitude of the received gradient, \( \nabla_t^{avl} \), is bounded above by a constant \( G \), i.e. $|\nabla_t^{avl}| \leq G.$
	\end{assumption}
	\begin{assumption}\label{asm:a4}
		The error threshold does not exceed the true gradient value, $\alpha < |\nabla_t|.$
	\end{assumption}
	The reason required to make assumption A\ref{asm:a4} is detailed in appendix \ref{sec:a4}.
	% The previous assumption A~\ref{asm:a4} is made to ensure that $\nabla_t^{avl}, \nabla_t$ are of same sign. Otherwise, the direction of received gradient would be opposite of the true one and no algorithm can move in the right direction. \pc{Shall we just say rationale behind making A4 is provided in the appendix. and in the appendix we can explain more by taking counterexamples and stuff?} \\ 
	In this setting, input to A-OBD, $\delta_t$, depends on $\nabla_{t}^{avl}$ as given in Theorem~\ref{Theorem:UB-noisy}.
	\begin{theorem}\label{Theorem:UB-noisy}
		Under the assumptions A\ref{asm:a1}, A\ref{asm:a2}, A\ref{asm:a3}, A\ref{asm:a4}, in the noisy fresh gradient setting \ref{setting:noisy-grad}, the competitive ratio of \ALGone ($\delta_{t}$) algorithm satisfies
		\begin{align}
			\mu_{\text{A-OBD} (\delta_{t})} &\leq \max \left\{1+\frac{2+\hat{b}}{\hat{a}},1+\hat{b}\right\}= O (M+L+\alpha^2), \nonumber
		\end{align}
		where
			$\hat{a} = \left(L-\frac{\alpha L}{G+L-2\alpha}\right)\left(1+\frac{\alpha}{G+L-\alpha}\right)$,\\
			$\hat{b} = \left(G+L+\frac{M}{2}+\frac{\alpha L}{G+L}\right)\left(1+\frac{\alpha}{G+L-\alpha}\right)$ and \\ $\delta_t = G+L-\nabla^{avl}_t$. 
	\end{theorem}
	%	Theorem \ref{Theorem:UB-noisy} is proved using the potential function method similar to previous sections. 
	% Define the hitting cost as $\mathcal{H}_t := f_t(x_{t}^{rec})$ and movement cost as $ \mathcal{M}_t := |x_t^{rec} - x_{t-1}^{rec}|$. 
	To prove \ref{Theorem:UB-noisy},  we derive bounds on the hitting cost $\mathcal{H}_t$ in terms of the switching cost $ \mathcal{M}_t$ while accounting for the noise in the gradient value. To address uncertainty in gradient estimates, we use dynamic $\delta_{t}$ values changing at every time step according to the value of the gradient estimate.
	Specifically, we analyze two extreme scenarios that the gradient value can take and use them to establish constants $a,b$ (as stated in Theorem \ref{Theorem:UB-noisy}) such that $a\mathcal{M}_t \leq \mathcal{H}_t \leq b\mathcal{M}_t$. Then we use the same potential function method as used to prove  Theorem~\ref{Theorem:UB-fresh}. A detailed proof is provided in appendix \ref{proof:Theorem-4}. 
	
	Note that the competitive ratio for \ALGone in noisy gradient setting grows at most with order $M+L+\alpha^2$. The dependence on $M$ arises due to similar reasons presented in fresh gradient setting, while $L$ and $\alpha$ arise because of the error due to noisy gradient. 
	%The quadratic dependence on $\alpha$..\pc{needs to be filled}\\
	
	We now establish a lower bound for any online algorithm operating in the noisy fresh gradient setting.
	\begin{theorem}\label{Theorem:LB-noisy}
		Under the assumptions A\ref{asm:a1}, A\ref{asm:a2}, A\ref{asm:a3}, A\ref{asm:a4}, in the noisy fresh gradient setting \ref{setting:noisy-grad}, any online algorithm must have a competitive ratio at least 
		$\Omega( \alpha)$.
		%		 [XXX: Do not understand this, shouldnt this be a function of $\alpha$?]. 
	\end{theorem}
	% We will analyze this case in the fresh gradient setting which can be extended to stale gradients using the method in previous section.\\ 
	We obtain Theorem \ref{Theorem:LB-noisy} using the same approach used to derive Theorem \ref{Theorem:LB-fresh} for the fresh gradient setting. 

	\section{Simulations}
	In this section we perform simulations to see the relationship between total cost of \ALGone with respect to $M$ and $L$ in different settings \ref{setting:fresh-grad},\ref{setting:stale-grad} and \ref{setting:noisy-grad}. 
	% [XXX: Describe what algorithms you are considering for comparison, and why?] 
	Let $\mathcal{S} = \{f_1,f_2,f_3,f_4,f_5\}$ be the set of functions given by
	% [XXX: Define set $\cal S$ as the 5 functions written in line as equations, and say $f_t \in \cal S$ is chosen uniformly at random for each $t$. ]
	\begin{align*}
		f_1(x) &= 0.5 M (x - 0.1)^2, \\
		f_2(x) &= 0.3 M (x - 0.3)^2, \\
		f_3(x) &= M (x - 0.2)^2, \\
		f_4(x) &= M (x - 0.8)^2, \\
		f_5(x) &= M (x - 0.4)^2.
	\end{align*}
	For each simulation, we set $T=100$ and $f_t \in \cal S$ is chosen uniformly at random for each $t$. The algorithm in all the settings run on the same sequence of functions for various values of $M, L\in [2,50]$. 
	The total cost incurred is plotted as a function of $M$ and $L$ respectively. Note that for the analysis with respect to $L$, we only consider the cost incurred in the stale and noisy gradient setting, as from our theoretical results, the cost incurred in the fresh gradient setting is independent of $L$. Since $C_{OPT}$ is independent of $M$ and $L$, our analysis provides insights on the growth of $\mu_{\text{A-OBD}}$. The findings, as illustrated in Figure \ref{fig:total_cost-vs-M} and \ref{fig:total_cost-vs-L}, are discussed in detail.
	% and discuss the findings you make in Fig. 1. 
	
	%    We consider $T=100$ and for each $t$, a function $f_t \in \mathcal{S}$ is chosen uniformly at random. 
	% We run the all the algorithms on same order of input functions for different values of $M$ ranging from $[2,50]$ and plot the total cost of the algorithm with respect to $M.$
	\begin{figure}[H]
		\centering
		\includegraphics[width=0.48\textwidth]{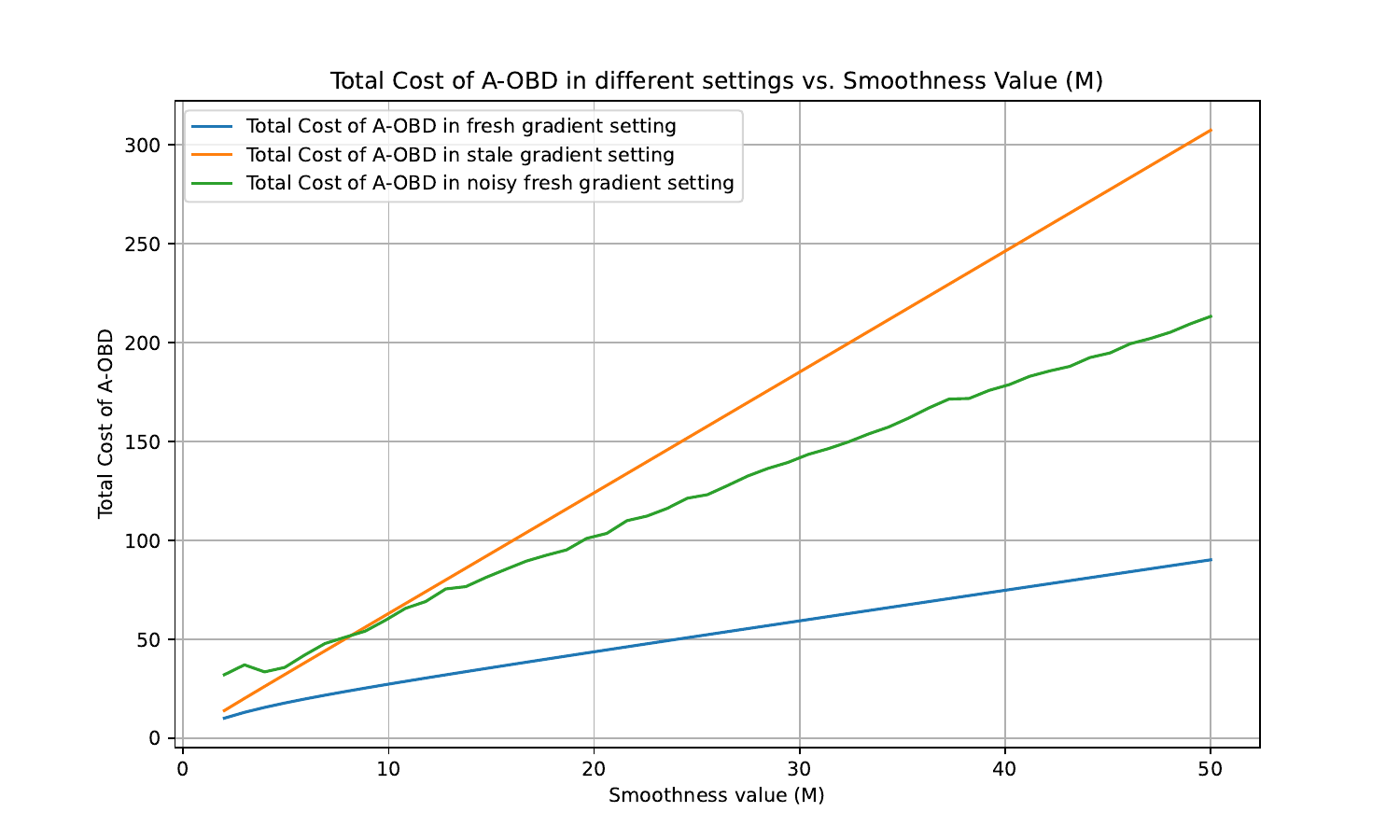}
		\caption{Total cost of A-OBD plotted as a function of $M$ for fresh, stale and noisy gradient settings \ref{setting:fresh-grad},\ref{setting:stale-grad} and \ref{setting:noisy-grad} respectively.}
		\label{fig:total_cost-vs-M}
		% \caption{XXX:Missing, write out what is being simulated under what settings etc.}
	\end{figure}
	
	\begin{figure}[H]
		\centering
		\includegraphics[width=0.48\textwidth]{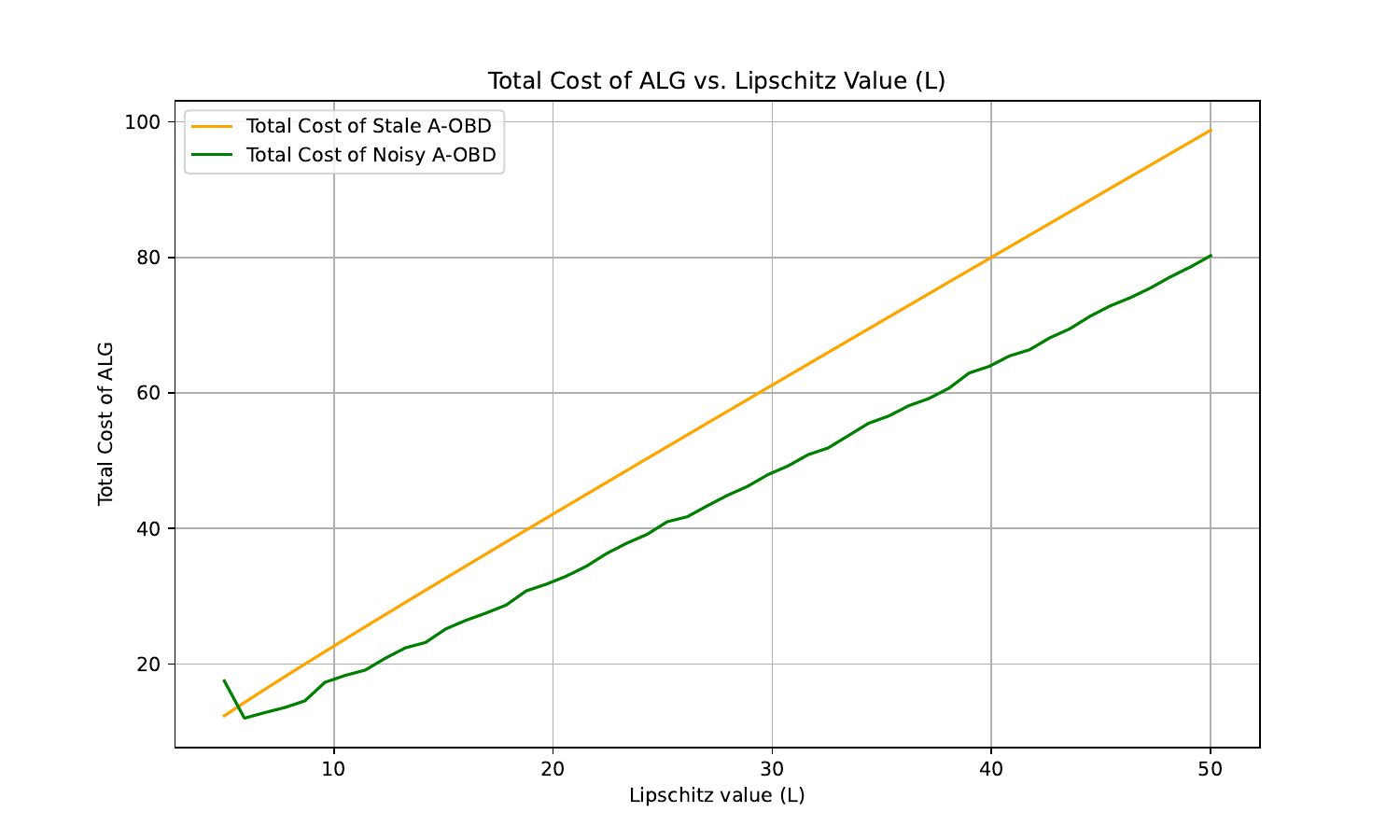}
		\caption{Total cost of A-OBD plotted as a function of $L$ for stale and noisy gradient settings \ref{setting:stale-grad} and \ref{setting:noisy-grad} respectively.}
		\label{fig:total_cost-vs-L}
		% \caption{XXX:Missing, write out what is being simulated under what settings etc.}
	\end{figure}
	
	\textbf{Discussion: }Simulation results indicate that the total cost of A-OBD $C_{\text{A-OBD}}$ scales linearly with $M$ in the fresh gradient setting \ref{setting:fresh-grad}, and linearly with both $M$ and $L$ in the stale gradient setting \ref{setting:stale-grad}. In the noisy gradient setting \ref{setting:noisy-grad}, although the total cost exhibits an overall linear trend with respect to $M$ and $L$, the noise in gradient values introduces slight perturbations, preventing the plot from being a perfect straight line. Since $C_{OPT}$ is independent of $M$ and $L$, these findings imply that the competitive ratio $\mu_{\text{A-OBD}}$ follows the same trend as $C_{\text{A-OBD}}$. Consequently, the simulations validate our theoretical results that in the frugal information setting the competitive ratio scales as $\Theta \left(max\{M,L\}\right)$.
	
	% From the simulations, we notice that the total cost of A-OBD grows linearly with $M$ in the fresh gradient setting \ref{setting:fresh-grad} and linearly in both $M, L$ in the stale and noisy gradient settings \ref{setting:stale-grad}, \ref{setting:noisy-grad} respectively. These results imply that $\mu_{\text{A-OBD}}$ follows the same trend. Thus, the simulations prove that the theoretical results obtain are valid and the competitive ratio in the frugal information setting scales $\Theta(max\{M,L\}$).
	
	%
	%\bibliography{References}
	
	\section{Conclusions}
	In this paper, we have expanded the horizon of online convex optimization with switching cost (OCO-S) with $\sfd=1$ under the frugal information setting, when only a single gradient evaluation is available. 
	Compared to prior work, where OCO-S has been considered under full information, i.e. full function over the full domain is known, 
	frugal information is almost a negligible requirement that is very useful from a practical point of view. Thus necessitating  the characterization of the  performance of online algorithms under the frugal setting. Compared to the full information setting where the competitive ratio is a constant, with frugal information setting, we show that the competitive ratio scales $\Theta(\max\{M,L\})$, where $M$ is the smoothness and Lipschitz parameter. Thus, frugal information does lead to some degradation in performance, however, that is limited by the parameters of the problem. We further studied the noisy gradient case in the frugal information setting, where the only available gradient evaluation is potentially noisy, and characterized that the competitive ratio of any online algorithm scales at least linearly in the maximum error and derived an online algorithm whose competitive ratio scales at most quadratically in the maximum error.
	
	\appendices
	\section{Proof of Theorem~\ref{Theorem:UB-fresh}} \label{proof:theorem-1}
	The proof is a simple modification of method used by \cite{bansal2015} to derive upper bounds in this setting. Let $x_t$ and $x_t^{OPT}$ represent the actions chosen by A-OBD$(\delta_1)$ and OPT, respectively, at time $t$. Define the potential function as $\phi(x_t,x_t^{OPT}):=\gamma|x_t-x_t^{OPT}|$.   
	Define $\mathcal{H}_t:=f_t(x_t)$ and $\mathcal{H}_t^{OPT}:=f_t(x_t^{OPT})$, $\mathcal{M}_t:=|x_t-x_{t-1}|$ and $\mathcal{M}_t^{OPT}:=|x_t^{OPT}-x_{t-1}^{OPT}|$, and $\Delta\phi_t:=\phi(x_t,x_t^{OPT})-\phi(x_{t-1}, x_{t-1}^{OPT})$.\\	
	By applying the triangle inequality, the following two inequalities are obtained.
	\begin{equation}
		\begin{gathered}
			\begin{aligned}
				\phi(x_{t-1},x_{t}^{OPT})-\phi(x_{t-1},x_{t-1}^{OPT}) \leq \gamma|x_t^{OPT}-x_{t-1}^{OPT}|
				\label{ineq:i1}
			\end{aligned}
		\end{gathered}
	\end{equation}
	
	\begin{equation}
		\begin{gathered}
			\begin{aligned}
				\phi(x_{t},x_{t}^{OPT})-\phi(x_{t-1},x_{t}^{OPT}) \leq \gamma|x_t-x_{t-1}|
				\label{ineq:i2}
			\end{aligned}
		\end{gathered}
	\end{equation}
	Combining inequalities \eqref{ineq:i1} and \eqref{ineq:i2}, we derive the following bound for $\Delta \phi_t$.
	\begin{equation}
		\begin{gathered}
			\begin{aligned}
				\Delta \phi_t \leq \gamma \mathcal{M}_t + \gamma \mathcal{M}_t^{OPT}
				\label{ineq:delphi_gen}
			\end{aligned}
		\end{gathered}
	\end{equation}
	For the case where $\mathcal{H}_t > \mathcal{H}_t^{OPT}$, inequality \eqref{ineq:i2} can be refined, subsequently improving \eqref{ineq:delphi_gen}. By assumption A~\ref{asm:a1}, $f_t(x^{*}_t)=0$ which implies, $x_t \in [x_{t-1},x^*_t].$ Since $f_t$ is convex and $x_t^*$ is its minimizer, it follows that $f_t(x_{t-1}) \geq f_t(x_t).$ Using continuity of $f_t$ and $f_t(x_{t-1}) \geq f_t(x_t) \geq f_t(x_t^{OPT}),$ we get that $x_t \in \left[x_{t-1}, x_t^{OPT}\right]$. This implies that  $\phi(x_{t},x_{t}^{OPT})-\phi(x_{t-1},x_{t}^{OPT}) \leq -\gamma|x_t-x_{t-1}|.$ Using this refined bound, we get the following lemma, equivalent to inequality \eqref{ineq:delphi_gen}.
%	\hs{Write above more precisely without using 'cross' and 'in between'}
	\begin{lemma}\label{lemma:delphi2}
		At time step $t$, when $\mathcal{H}_t>\mathcal{H}_t^{OPT}$, we have $\Delta \phi_t \leq \gamma \mathcal{M}_t^{OPT}-\gamma \mathcal{M}_t$
	\end{lemma}
	Now we bound $\mathcal{H}_t$ from both directions in terms of $\mathcal{M}_t$ in the following lemmas.
	\begin{lemma}\label{lemma:lb-cost}
		Under the assumption A\ref{asm:a1}. in the fresh gradient setting (a), we have $\mathcal{H}_t \geq \delta \mathcal{M}_t$.
	\end{lemma}
	\begin{proof}
		Since $f_t$ is convex, the gradient line $L_1$ lies below the function $f_t$, i.e. $L_1(x)\le f_t(x)$ for all $x$. Consequently, the intersection of $L_1$ with $y=\delta|x-x_{t-1}|$ is similarly bounded above by $f_t(x_t)$. Hence $f_t(x_t) \geq \delta|x_t-x_{t-1}|=\delta \mathcal{M}_t.$
	\end{proof}
	\begin{lemma}\label{lemma:ub-cost}
		Under the assumption A\ref{asm:a1}. in setting (a), we have $\mathcal{H}_t \leq (\delta+M/2)\mathcal{M}_t$
	\end{lemma}
	\begin{customproof}[Lemma \ref{lemma:ub-cost}]
		Using the $M$-smoothness property of $f_t$ we have $f_t(x_t) - \left\{ f_t(x_{t-1}) + \left\langle \nabla f_t(x_{t-1}), (x_t - x_{t-1}) \right\rangle \right\}  \leq \frac{M}{2} |x_t - x_{t-1}|^2$. 
		\begin{equation}\nonumber
			\begin{aligned}
				\Rightarrow \quad & \mathcal{H}_t - \delta \mathcal{M}_t \stackrel{(a)}{\leq} \frac{M}{2} |x_t - x_{t-1}|^2 \stackrel{(b)}{\leq} \frac{M}{2} |x_t - x_{t-1}| \\ 
				%\Rightarrow \quad & \mathcal{H}_t - \delta \mathcal{M}_t \stackrel{(b)}{\leq} \frac{M}{2} |x_t - x_{t-1}| \\  
				\Rightarrow \quad & \mathcal{H}_t \leq \left(\delta + \frac{M}{2} \right) \mathcal{M}_t
			\end{aligned}
		\end{equation}
		where (a) follows since $f_t(x_{t-1})+\left <\nabla f_t\left(x_{t-1}\right),\left(x_t-x_{t-1}\right)\right>$ is the value of the gradient line evaluated at the point where it intersects $y=\delta|x-x_{t-1}|$, which is equal to $\delta|x_t-x_{t-1}|=\delta \mathcal{M}_t$, while (b) follows since $|x_t-x_{t-1}| \leq 1.$\\
		This completes the proof of Lemma \ref{lemma:ub-cost}.
	\end{customproof}
	
	%At time $t$, the approach is to bound the hitting cost $\mathcal{H}_t$ of the function $f_t$ in terms of the movement cost $\mathcal{M}_t$. Following this, we would use the existing method of \cite{bansal2015} to upper bound the competitive ratio. The analysis of setting \ref{setting:fresh-grad} is crucial, while other settings are minor modifications of \ref{setting:fresh-grad}.
	
	Using Lemmas \ref{lemma:lb-cost} and \ref{lemma:ub-cost}, we will derive an upper bound for the term \( \mathcal{H}_t + \mathcal{M}_t + \Delta \phi_t \), which aids in establishing an upper bound on the competitive ratio. We analyze this bound in two cases based on the relationship between the hitting cost of the algorithm and the optimum. 
	
	\textbf{Case 1: $\mathcal{H}_t>\mathcal{H}_t^{OPT}$}
	\begin{equation}\nonumber
		\begin{gathered}
			\begin{aligned}
				\mathcal{H}_t+\mathcal{M}_t+\Delta \phi_t &\stackrel{(a)}\leq \left(\delta + \frac{M}{2}\right)\mathcal{M}_t+\mathcal{M}_t+\Delta \phi_t\\
				&\stackrel{(b)}\leq  \left(\delta + \frac{M}{2}+1\right)\mathcal{M}_t - \gamma \mathcal{M}_t + \gamma \mathcal{M}_t^{OPT}\\
				&\stackrel{(c)}\leq \gamma \left(\mathcal{M}_t^{OPT}+\mathcal{H}_t^{OPT}\right) 
			\end{aligned}
		\end{gathered}
	\end{equation}
	where (a) follows from Lemma \ref{lemma:ub-cost}, (b) uses Lemma \ref{lemma:delphi2}, and (c) holds by setting \( \gamma = \delta + \frac{M}{2} + 1 \) and adding an additional \( \gamma \mathcal{H}_t^{OPT} \) term.
	
	\textbf{Case 2: $\mathcal{H}_t \leq \mathcal{H}_t^{OPT}$}
	\begin{equation}\nonumber
		\begin{gathered}
			\begin{aligned}
				\mathcal{H}_t+\mathcal{M}_t+\Delta \phi_t &\stackrel{(d)}\leq \mathcal{H}_t + \frac{\mathcal{H}_t}{\delta}+\Delta \phi_t\\
				&\stackrel{(e)}\leq \mathcal{H}_t \left(1+\frac{1}{\delta}\right)+\gamma \mathcal{M}_t + \gamma \mathcal{M}_t^{OPT}\\
				&\stackrel{(f)}\leq \mathcal{H}_t^{OPT}\left(1+\frac{1}{\delta}+\frac{\gamma}{\delta}\right)+\gamma \mathcal{M}_t^{OPT}\\
				&\leq \max \left\{\left(1+\frac{1}{\delta}+\frac{\gamma}{\delta}\right),\gamma\right\}\left(\mathcal{H}_t^{OPT}+\mathcal{M}_t^{OPT}\right)
			\end{aligned}
		\end{gathered}
	\end{equation}
	where (d) follows from Lemma \ref{lemma:lb-cost}, (e) is derived using \eqref{ineq:delphi_gen} obtained above, and (f) holds because \( \mathcal{H}_t \leq \mathcal{H}_t^{OPT} \).\\
	Combining the bounds derived for both cases, we have
	\begin{equation}\nonumber
		\begin{gathered}
			\begin{aligned}
				\mathcal{H}_t + \mathcal{M}_t + \Delta \phi_t \leq \max\left\{\delta + 1 + \frac{M}{2}, 2 + \frac{2}{\delta} + \frac{M}{2\delta}\right\}(\mathcal{H}_t^{OPT} + \mathcal{M}_t^{OPT}). \label{eq:ub-fresh}
			\end{aligned}
		\end{gathered}
	\end{equation} 
	Finally, noting that \( \sum_{t=1}^T \Delta \phi_t \geq 0 \), we establish an upper bound on the competitive ratio as follows. $C_{\text{A-OBD}(\delta_1)}$
	\begin{equation}\nonumber
		\begin{gathered}
			\begin{aligned}
				&= \sum_{t=1}^T \left(\mathcal{H}_t + \mathcal{M}_t\right) \\
				&\leq \sum_{t=1}^T \left(\mathcal{H}_t + \mathcal{M}_t + \Delta \phi_t\right) \\
				&\stackrel{\eqref{eq:ub-fresh}}\leq \sum_{t=1}^T \max\left\{\delta + 1 + \frac{M}{2}, 2 + \frac{2}{\delta} + \frac{M}{2\delta}\right\} \left(\mathcal{H}_t^{OPT} + \mathcal{M}_t^{OPT}\right) \\
				&= \max\left\{\delta + 1 + \frac{M}{2}, 2 + \frac{2}{\delta} + \frac{M}{2\delta}\right\} \sum_{t=1}^T \left(\mathcal{H}_t^{OPT} + \mathcal{M}_t^{OPT}\right) \\
				&= \max\left\{\delta + 1 + \frac{M}{2}, 2 + \frac{2}{\delta} + \frac{M}{2\delta}\right\} C_{OPT}.
			\end{aligned}\
		\end{gathered}
	\end{equation}
	
	Hence, the competitive ratio is bounded by
	\begin{align*}
		\mu_{\text{A-OBD}(\delta_1)} \leq \max\left\{\delta + 1 + \frac{M}{2}, 2 + \frac{2}{\delta} + \frac{M}{2\delta}\right\}.
	\end{align*}
	We can now obtain an optimal choice of  $\delta$  to minimize the derived upper bound. Choosing  $\delta = {\hat \delta} = \frac{1}{2} \sqrt{\left(\frac{M}{2} + 1\right)^2 + 7} - \frac{M}{4} + \frac{1}{2}$ gives us the following result.
	\begin{equation}\nonumber
		\begin{gathered}
			\begin{aligned}
				\mu_{\text{A-OBD}( {\hat \delta} )} \leq \frac{M}{4}+\frac{3}{2}+\frac{1}{2}\sqrt{\left(\frac{M}{2}+1\right)^2+7}.
			\end{aligned}
		\end{gathered}
	\end{equation}
	This completes the proof of Theorem \ref{Theorem:UB-fresh}.
	
	\section{Proof of Theorem~\ref{Theorem:LB-fresh}} \label{proof:Theorem-2}
	
%	Define the set of functions, \textit{ALG-Id($f^{\text{value}},\nabla^{\text{value}}$)}, that have the same function evaluation ($f^{\text{value}}$) and gradient ($\nabla^{\text{value}}$) values at \( x_0 \).
%	By carefully designing these functions, we demonstrate that the lower bound on the performance of any algorithm is at least \(\Omega(M)\) showing optimality of A-OBD{$(\delta_1)$}. The key idea used to prove Theorem~\ref{Theorem:LB-fresh} is to adversarially select different functions from the set \textit{ALG-Id} according to the action of the online algorithm. This construction ensures that no algorithm can surpass the derived lower bound, as the fundamental challenge lies in distinguishing between the functions using only frugal information.
	Consider the OCO-LS problem instance in the fresh gradient setting where $T=1, x_0=0$ and the received information to an online algorithm is $f_1(0)=\frac{\nabla}{M},\nabla f_1(0)=\nabla$ where $\nabla>3.$ Consider the following functions:
	\begin{equation*} \label{eq:lb-functions}
		\begin{aligned}
			g_1(x) &= 
			\begin{cases} 
				-\nabla x + \frac{\nabla^2}{M}, & \text{if } x \leq \frac{\nabla}{2M}, \\[10pt]
				\frac{M}{2} \left( x - \frac{1.5 \nabla}{M} \right)^2, & \text{otherwise}.
			\end{cases} \\[15pt]
			g_2(x) &= 
			\begin{cases} 
				\frac{M}{2} \left( x - \frac{\nabla}{M} \right)^2 + \frac{\nabla^2}{2M}, & \text{if } x \leq \frac{\nabla}{2M}, \\[10pt]
				\frac{\nabla}{2} \left( \frac{7 \nabla}{4M} - x \right), & \text{if } \frac{\nabla}{2M} \leq x \leq \frac{3 \nabla}{2M}, \\[10pt]
				\frac{M}{2} \left( x - \frac{2 \nabla}{M} \right)^2, & \text{otherwise}.
			\end{cases} \\[15pt]
			g_3(x) &= 
			\begin{cases} 
				\frac{M}{2} \left( x - \frac{\nabla}{M} \right)^2 + \frac{\nabla^2}{2M}, & \text{if } x \leq \frac{2 \nabla}{3M}, \\[10pt]
				\frac{\nabla}{3} \left( \frac{7 \nabla}{3M} - x \right), & \text{if } \frac{2 \nabla}{3M} \leq x \leq \frac{13 \nabla}{6M}, \\[10pt]
				\frac{M}{2} \left( x - \frac{2.5 \nabla}{M} \right)^2, & \text{otherwise}.
			\end{cases}
		\end{aligned}
	\end{equation*}
%	\begin{figure}
%		\centering
%		\includegraphics[width=0.5\textwidth]{General-LB}
%		\caption{set of functions $\mathcal{G}$}
%		\label{fig:General_LB}
%	\end{figure}
	
%	These functions are composition of quadratic, linear and quadratic component except the first one which just is linear and quadratic. Note these functions lie in set \textit{ALG-Id($f^{\text{value}},\nabla^{\text{value}}$)} with $f^{\text{value}}=\frac{\nabla^2}{M}$ and $\nabla^{\text{value}}=-\nabla.$
%	
%	Each function ultimately [XXX: not clear what ultimately means here] ends with a simple quadratic term of the form $\frac{M}{2}(x-c)^2$ with different $c$ values. Consider $T=1$. The adversary aims to input functions $g_1, g_2, g_3$ [XXX: at time $1$?] such that output [XXX: $x_1$?] of \ALG lies [XXX: not clear what lies means here] on the linear part of these functions or to the farther side of the minima. Keep $\nabla>3$ to ensure that the action of OPT will lie in the ending [XXX: no idea what that means, use variable $x_1^{OPT}$ and say where this will be] quadratic segment. 
	
	Let $\mathcal{G}=\{g_1,g_2,g_3\},$ which is a set of $M$-smooth functions that satisfy \( g_i(x_0) = g_j(x_0) \) and \( g_i'(x_0) = g_j'(x_0) \ \forall \ i,j \in \{1,2,3\} \). 
	%Essentially, any online algorithm \ALG cannot differentiate between functions from the set $\mathcal{G}.$ Exploiting this advantage, functions can be chosen from $\mathcal{G}$ adversarially for each action($x_1$) of \ALG.
	
	With $T=1$, $C_{OPT}(g_1)=\frac{3\nabla}{2M}-\frac{1}{2M},C_{OPT}(g_2)=\frac{2\nabla}{M}-\frac{1}{2M}$, $$C_{OPT}(g_3)=\frac{5\nabla}{2M}-\frac{1}{2M},$$ respectively. Next, we will divide the number line into the following intervals: $${\cal I}_1=\left[0,\frac{\nabla}{2M}\right],{\cal I}_2=\left[\frac{\nabla}{2M},\frac{3\nabla}{2M}\right],$$ $$ {\cal I}_3=\left[\frac{3\nabla}{2M},\frac{13\nabla}{6M}\right],{\cal I}_4=\left[\frac{13\nabla}{6M},1\right].$$
	
	Let $x_1$ be the action chosen by any online algorithm $\cA$ for time $1$. Then if $x_1 \in \{{\cal I}_1,{\cal I}_2,{\cal I}_3,{\cal I}_4\}$, $f_1$ is selected as \( \{g_1, g_2, g_3, g_1 \} \), respectively. This implies that
	%For each interval $\mathcal{I}_i \ \forall \ i$, calculate action $x_1 \in \mathcal{I}_i$ with minimum total cost to provide a lower bound on $C_\mathcal{A}$ for any $\cA$. Using these lower bound on the total cost we get respective lower bounds on the competitive ratio for each function	
	$\mu_{\cA}(g_1)=\frac{\nabla}{3}, \mu_{\cA}(g_2)= \frac{\nabla}{16}, \mu_{\cA}(g_3)= \frac{\nabla}{45}, \mu_{\cA}(g_1) =\frac{4\nabla}{27}$.

	Since by definition, $\mu_\cA$ is the maximum over all possible inputs, we get $\mu_{\cA}\ge  \frac{\nabla}{45} $
	%From this, we can conclude that the lower bound for the performance of any \ALG in this setting is at least \( \frac{\nabla}{45} \). 
	For our specific setting, we require that the minimizer of loss functions to lie within the interval \( [0, 1] \). To ensure this, we must have $\frac{5\nabla}{2M} \leq 1,$ which implies that we can choose the maximum value of $\nabla$ as $\frac{2M}{5}.$ Using this particular value, we obtain the lower bound $\mu_{\mathcal{A}} \geq \frac{2M}{135}$ for any online algorithm $\mathcal{A}.$

	\section{Proof of Theorem~\ref{Theorem:UB-stale}} \label{proof:Theorem-3}
	For a fixed input sequence ${\cal F} = \{f_t\}_{t=1}^T$, let $C_{\mathcal{A}}^{fresh}$ and $C_{\mathcal{A}}^{stale}$ be the costs incurred by the algorithm $\mathcal{A}$ in the fresh and stale gradient setting, 
	respectively.
	% Our proof relies on the relationship between $C_{\mathcal{A}}^{fresh}$ and $C_{\mathcal{A}}^{stale}$. We will first derive this relationship and then proceed to prove the upper bound for \ALGone ($\delta_{2}$) in the stale gradient setting. 
	\begin{lemma}\label{lemma:cost-relation}
		%Let $C_{\mathcal{A}}^{fresh}$ and $C_{\mathcal{A}}^{stale}$  denote the costs incurred by the algorithm  $\mathcal{A}$  in the fresh and stale gradient settings over the same input sequence, respectively. Then, 
		
		\begin{equation}
			\begin{gathered}
				\begin{aligned}
					C_{\mathcal{A}}^{stale} &\leq C_{\mathcal{A}}^{fresh} + \sum_{t=1}^{T} L\mathcal{M}_t = \sum_{t=1}^{T}\left(\mathcal{H}_t + \left(L+1\right)\mathcal{M}_t\right)
				\end{aligned}
			\end{gathered}
		\end{equation}
		where, $L$ is the Lipschitz constant for the function $f_t$ and $\mathcal{H}_t, \mathcal{M}_t$ are the hitting cost and switching cost respectively. 
	\end{lemma}
	%\begin{customproof}[Lemma \ref{lemma:cost-relation}]
	\begin{proof}
		For a fixed input sequence ${\cal F} =\{f_t\}_{t=1}^T$, let the actions of algorithm A-OBD($\delta_2$) in the fresh gradient setting be  $x_t^{fresh}$. Then actions of  algorithm A-OBD($\delta_2$) with ${\cal F}$ in the stale gradient setting are $x_t^{stale} = x_{t-1}^{fresh}$. Thus, the cost for A-OBD($\delta_2$) in the fresh and stale setting for ${\cal F}$ are: $C_{\mathcal{A}}^{fresh} = \sum_{t=1}^{T}f_t(x_{t}) + \sum_{t=1}^{T}|x_{t}-x_{t-1}|$,  and $
			C_{\mathcal{A}}^{stale}  = \sum_{t=1}^{T}f_t(x_{t-1}) + \sum_{t=1}^{T-1}|x_{t}-x_{t-1}|$,	where we have dropped the superscript $fresh$ from  $x_t^{fresh}$ for notational simplicity.
		Thus, $C_{\mathcal{A}}^{stale} - C_{\mathcal{A}}^{fresh} $
		\begin{align*}
			&\le \sum_{t=1}^{T}|f_t(x_{t})-f_t(x_{t-1})| - |x_{T}-x_{T-1}|,\\
			& \stackrel{A\ref{asm:a2}}\le \sum_{t=1}^{T} L| x_{t}-x_{t-1}| - |x_{T}-x_{T-1}| \le L\mathcal{M}_t^{ALG}. 
		\end{align*}
		Thus, we get that 
		\begin{align*}
			C_{\mathcal{A}}^{stale} &\le C_{\mathcal{A}}^{fresh} + L\mathcal{M}_t^{ALG}\\
			& \le 
			\sum_{t=1}^{T}\left(\mathcal{H}_t^{ALG} + \left(L+1\right)\mathcal{M}^{ALG}_t\right).
		\end{align*}
		
	\end{proof}
	%\begin{equation}
	%			%\begin{gathered}
	%				\begin{aligned}
		%					 \nonumber \\
		%					%    C_{OPT} &= \sum_{t=1}^{T}f_t(x_{t}^{OPT}) + \sum_{t=1}^{T}|x_{t}^{OPT}-x_{t-1}^{OPT}| \nonumber \\
		%					C_{\mathcal{A}}^{stale} - C_{\mathcal{A}}^{fresh} &= \sum_{t=1}^{T}(f_t(x_t^{ALG}) - f_t(x_{t-1}^{ALG})) - |x_T^{ALG}-x_{T-1}^{ALG}| \stackrel{A\ref{asm:a2}}{\leq} \sum_{t=1}^{T} L \mathcal{M}_t^{ALG} \nonumber \\
		%					\Rightarrow C_{\mathcal{A}}^{stale} &\leq C_{\mathcal{A}}^{fresh} + \sum_{t=1}^{T} L \mathcal{M}_t^{ALG} = \sum_{t=1}^{T} \left(\mathcal{H}_t^{ALG} + (L+1)\mathcal{M}_t^{ALG}\right) \nonumber
		%				\end{aligned}
	%			%\end{gathered}
	%		\end{equation}
	%		This completes the proof of Lemma \ref{lemma:cost-relation}.
	
	%	\begin{lemma}\label{lemma:opt-delta-stale}
	%		In the stale gradients setting, the optimal value of $\delta$ that minimizes $\max \left \{\delta+L+1+\frac{M}{2},2+\frac{2(L+1)}{\delta}+\frac{M}{2\delta}\right\}$
	%		is given by
	%		\begin{align*}
		%			\delta^*_{stale} =  \frac{1}{4} \sqrt{(M+2L-2)^2-8(M+4L+4)} - \frac{M}{4} - \frac{L}{2} + \frac{1}{2}.
		%		\end{align*}
	%	\end{lemma} 
	%	Proof of lemma \ref{lemma:opt-delta-stale} can be found in appendix \ref{proof:lemma-opt-delta-stale}. Lemma \ref{lemma:opt-delta-stale} provides the foundation for Theorem \ref{Theorem:UB-stale} and its proof.
	
	With this relationship established through Lemma \ref{lemma:cost-relation}, we are ready to proceed with the proof of Theorem \ref{proof:Theorem-3}. Analogous to the fresh setting, using lemmas \ref{lemma:lb-cost} and \ref{lemma:ub-cost}, we derive an upper bound for the term $\mathcal{H}_t+(L+1)\mathcal{M}_t+\Delta \phi_t$ by considering two cases: $\mathcal{H}_t>\mathcal{H}_t^{OPT}$ and $\mathcal{H}_t\leq \mathcal{H}_t^{OPT}$. \\
	\newline
	\textbf{Case 1. $\mathcal{H}_t>\mathcal{H}_t^{OPT}$}
	\begin{equation} \nonumber
	\begin{gathered}
		\begin{aligned}
			\mathcal{H}_t+(L+1)\mathcal{M}_t+\Delta \phi_t &\stackrel{(a)}\leq \left(\delta + \frac{M}{2}\right)\mathcal{M}_t+\left(L+1\right)\mathcal{M}_t+\Delta  \phi_t\\
			&\stackrel{(b)}\leq  \left(\delta + \frac{M}{2}+L+1\right)\mathcal{M}_t - \gamma \mathcal{M}_t + \gamma \mathcal{M}_t^{OPT}\\
			&\stackrel{(c)}\leq \gamma\left(\mathcal{M}_t^{OPT}+\mathcal{H}_t^{OPT}\right) 	
		\end{aligned}
	\end{gathered}
	\end{equation}
	where, (a) follows from Lemma \ref{lemma:ub-cost}, (b) follows from Lemma \ref{lemma:delphi2}, and (c) holds by choosing $\gamma=\delta + \frac{M}{2}+L+1$ and adding an additional $\gamma \mathcal{H}_t^{OPT}$ term.\\
	
	\textbf{Case 2. $\mathcal{H}_t<\mathcal{H}_t^{OPT}$}
	\begin{equation}\nonumber
	\begin{gathered}
		\begin{aligned}
			\mathcal{H}_t+(L+1)\mathcal{M}_t+\Delta \phi_t &\stackrel{(d)}\leq \mathcal{H}_t\left(1+\frac{L+1}{\delta}\right)+\Delta \phi_t\\
			&\stackrel{(e)}\leq \mathcal{H}_t\left(1+\frac{L+1}{\delta}\right)+\gamma \mathcal{M}_t + \gamma \mathcal{M}_t^{OPT}\\
			&\stackrel{(f)}\leq \mathcal{H}_t^{OPT}\left(1+\frac{L+1}{\delta}+\frac{\gamma}{\delta}\right)+\gamma \mathcal{M}_t^{OPT}\\
			&\leq \max \left\{\left(1+\frac{L+1}{\delta}+\frac{\gamma}{\delta}\right),\gamma \right\}\left(\mathcal{H}_t^{OPT}+\mathcal{M}_t^{OPT}\right)
		\end{aligned}
	\end{gathered}
	\end{equation}
	where (d) follows Lemma \ref{lemma:lb-cost}, (e) follows from \eqref{ineq:delphi_gen}, and (f) holds because $\mathcal{H}_t \leq \mathcal{H}_t^{OPT}.$\\
	
	% [XXX: Where is Lemma 4 being used?]
	Combining the bounds derived for both cases, we have $\mathcal{H}_t+(L+1)\mathcal{M}_t+\Delta \phi_t$
	\begin{equation} 
	\begin{gathered}
		\begin{aligned}
			\leq \max \left\{ \delta + L + 1 + \frac{M}{2}, 2 + \frac{2(L+1)}{\delta} + \frac{M}{2\delta} \right\}\left(\mathcal{H}_t^{OPT}+\mathcal{M}_t^{OPT}\right).\label{eq:ub-stale}
		\end{aligned}
	\end{gathered}
	\end{equation}
	Finally, utilizing \eqref{eq:ub-stale} and the fact that \( \sum_{t=1}^T \Delta \phi_t \geq 0 \), we establish an upper bound on the competitive ratio:
	\begin{equation}\nonumber
	\begin{gathered}
		\begin{aligned}
			C_{ALG} &= \sum_{t=1}^T \left(\mathcal{H}_t + \left(L+1\right)\mathcal{M}_t\right) \\
			&\leq \sum_{t=1}^T \left(\mathcal{H}_t+(L+1)\mathcal{M}_t+\Delta \phi_t\right) \\
			&\stackrel{\eqref{eq:ub-stale}}\leq \sum_{t=1}^T \max \left\{ \delta + L + 1 + \frac{M}{2}, 2 + \frac{2(L+1)}{\delta} + \frac{M}{2\delta} \right\} \left(\mathcal{H}_t^{OPT} + \mathcal{M}_t^{OPT}\right) \\
			&= \max \left\{ \delta + L + 1 + \frac{M}{2}, 2 + \frac{2(L+1)}{\delta} + \frac{M}{2\delta} \right\} \sum_{t=1}^T \left(\mathcal{H}_t^{OPT} + \mathcal{M}_t^{OPT}\right) \\
			&= \max \left\{ \delta + L + 1 + \frac{M}{2}, 2 + \frac{2(L+1)}{\delta} + \frac{M}{2\delta} \right\} C_{OPT}.
		\end{aligned}
	\end{gathered}
	\end{equation}
	
	Hence, the competitive ratio is bounded by
	\begin{equation}\nonumber
	\begin{gathered}
		\begin{aligned}
			\mu_{\ALGone(\delta)} \leq \max \left\{ \delta + L + 1 + \frac{M}{2}, 2 + \frac{2(L+1)}{\delta} + \frac{M}{2\delta} \right\}.
		\end{aligned}
	\end{gathered}
	\end{equation}
	We can now obtain an optimal choice of  $\delta$  to minimize the derived upper bound. Choosing
	\begin{equation}\nonumber
	\begin{gathered}
		\begin{aligned}
			{\tilde \delta} =  \frac{1}{4} \sqrt{(M+2L-2)^2-8(M+4L+4)} - \frac{M}{4} - \frac{L}{2} + \frac{1}{2} 
		\end{aligned}
	\end{gathered}
	\end{equation}
	gives us the following result. 
	\begin{equation}\nonumber
	\begin{gathered}
		\begin{aligned}
			\mu_{\ALGone({\tilde \delta})} \leq \frac{M}{4}+ \frac{L}{2} +\frac{3}{2} +\frac{1}{4} \sqrt{(M+2L-2)^2-8(M+4L+4)}.
		\end{aligned}
	\end{gathered}
	\end{equation}
	This completes the proof of Theorem \ref{Theorem:UB-stale}.
	
	\section{Rationale behind A\ref{asm:a4}} \label{sec:a4}
    Assumption \ref{asm:a4} requires that the noise tolerance $\alpha$ is strictly less than the magnitude of the true gradient, i.e., $\alpha < |\nabla_t|$. This condition ensures that $\nabla_t^{avl}$ reflects the correct descent direction for $f_t$. If $\alpha$ were to exceed $|\nabla_t|$, the noise could be large enough to reverse the intended update direction. In such a case, the noisy gradient could point away from the optimum causing the algorithm to update in the opposite direction and increase \ref{eq:total-cost}. By imposing $\alpha < |\nabla_t|$, we ensure that the deviation introduced by the noise is bounded, so that even in extreme scenario, $\nabla_t^{avl}$ remains aligned with $\nabla_t$. This guarantees that the algorithm consistently moves in a direction of minimizing \ref{eq:total-cost}.
	
	\section{Proof of Theorem~\ref{Theorem:UB-noisy} }\label{proof:Theorem-4} 
	Theorem \ref{Theorem:UB-noisy} establishes an upper bound on the competitive ratio for \ALGone ($\delta_t$) when the gradient information($\nabla_t^{avl}$) is noisy. We cannot use the same method as Appendix~\ref{proof:theorem-1} because Lemma~\ref{lemma:lb-cost}, \ref{lemma:ub-cost} do not hold necessarily. Hence, we derive new bounds $\hat{a} \mathcal{M}_t \leq \mathcal{H}_t \leq \hat{b} \mathcal{M}_t$ for some $\hat{a},\hat{b}>0.$ Our approach is to use extremal cases ($\nabla=\nabla_t - \alpha, \nabla_t + \alpha$) to obtain the values of $\hat{a},\hat{b}$ for any noisy gradient such that $|\nabla_t^{avl}-\nabla_{t}|\leq \alpha.$

	%[XXX: why are you writing Algorithm definition again as a subroutine. We wont have that much space. Use the definition of Algorithm]
	%\hs{The way we have defined the algorithm, it directly outputs using the function and gradient information. So need to define this kind of subroutine.}
	\begin{algorithm}[H]  
	\caption{Subroutine({$f^{\text{val}}, \nabla, \hat{x}, \delta$})} \label{alg:subroutine-AOBD}
	\begin{algorithmic}[1]
		\Function{Subroutine}{$f^{\text{val}}, \nabla, \hat{x}, \delta$}
		\State $L_1(x) \gets f^{\text{val}} + \nabla \cdot (x - \hat{x})$ 
		\State $L_2(x) \gets \delta \cdot \text{sign}(-\nabla) \cdot (x - \hat{x})$ 
		\State Find $x$ such that $L_1(x) = L_2(x)$ 
		\State \Return $x$ 
		\EndFunction
	\end{algorithmic}
	\end{algorithm}
	Let $x_t^{true}$ be the output of Subroutine({$f_t^{\text{avl}}, \nabla_t, x_{t-1}, \delta_t$}) with the choice of $\delta_t$ specified in the statement of Theorem~\ref{Theorem:UB-noisy}.
	
	Define $\underline{\nabla_t} := \nabla_t - \alpha, \overline{\nabla}_t := \nabla_t + \alpha$. Let $\underline{x_t}, \overline{x_t}$ be the output of Subroutine({$f_t^{\text{avl}}, \nabla, x_{t-1}, \delta_t$}) with $\nabla=\underline{\nabla_t},\overline{\nabla}_t$ as the input respectively with same $\delta_t$ as specified in Theorem~\ref{Theorem:UB-noisy}. We similarly define hitting and switching cost associated with $\underline{x_t},\overline{x_t}$ as $(\underline{\mathcal{H}_t}, \underline{\mathcal{M}_t})$ and $(\overline{\mathcal{H}_t}, \overline{\mathcal{M}_t})$ respectively.
	
	We will take $\sign(\nabla_{t})<0$ (similar proof can be done for $\sign(\nabla_{t})>0$). The sign of $\nabla_t$ determines the direction of movement of $x_t$ relative to $x_{t-1}$. Specifically $x_{t-1} \leq x_t \leq x_t^*$ when  $\nabla_t < 0,$ and $x_{t-1} \geq x_t \geq x_t^*$ when $\nabla_t > 0.$ For simplicity, we take $\nabla_t = |\nabla_t|$ and omit the absolute value notation throughout the analysis.

    From the above definitions, the following is immediate
	\begin{align}
	\underline{\mathcal{M}_t} &\leq \mathcal{M}_t \leq \overline{\mathcal{M}_t}, \label{ineq:UB-LB-Mt} \\
	\overline{\mathcal{H}_t}  &\leq \mathcal{H}_t \leq  \underline{\mathcal{H}_t}. \label{ineq:UB-LB-Ht}
	\end{align}
	%	Suppose there exist $a,b>0$ such that $\underline{\mathcal{H}_t} \leq b\underline{\mathcal{M}_t}$ and $\overline{\mathcal{H}_t} \geq a\overline{\mathcal{M}_t}$, we get the following bounds on $\mathcal{H}_t$.
	%	$$a\overline{\mathcal{M}_t} \leq \mathcal{H}_t \leq  b\underline{\mathcal{M}_t} \\ \Rightarrow a \mathcal{M}_t \leq \mathcal{H}_t \leq b \mathcal{M}_t.$$
	%	
	%	Our next step is to determine the values of $a$ and $b$.\\ From \eqref{eq:ulbhtmt}, we get
	%	%As the $\nabla_t$ also lies in between $(\underline{\nabla},\overline{\nabla})$ we have the following.
	%	\begin{equation}\nonumber
	%		\begin{gathered}
		%			\begin{aligned}
			%				%\underline{\mathcal{M}_t} &\leq \mathcal{M}_t^{true} \leq \overline{\mathcal{M}_t}\\ 
			%				%\overline{\mathcal{H}_t} &\leq \mathcal{H}_t^{true} \leq \underline{\mathcal{H}_t}\\
			%				\overline{\mathcal{H}_t} &\leq \left(\delta + \frac{M}{2}\right)\mathcal{M}_t^{true} \leq \left(\delta + \frac{M}{2}\right)\overline{\mathcal{M}_t}\\
			%				\Rightarrow \underline{\mathcal{H}_t} &\geq \delta \mathcal{M}_t^{true} \geq \delta \underline{\mathcal{M}_t}
			%			\end{aligned}
		%		\end{gathered}
	%	\end{equation}
	We can express the values of $\underline{\mathcal{M}_t},\overline{\mathcal{M}_t}$ using their definition in terms of $\mathcal{M}_t^{true},\nabla_{t},\delta_{t}$ and $\alpha,$ as shown below 
	\begin{equation}
	\overline{\mathcal{M}_t} = \mathcal{M}_t^{true} \frac{|\nabla_t|+\delta_t}{|\nabla_t|-\alpha+\delta_t}
	\label{eq:Mupline}
	\end{equation}
	\begin{equation}
	\underline{\mathcal{M}_t} = \mathcal{M}_t^{true} \frac{|\nabla_t|+\delta_t}{|\nabla_t|+\alpha+\delta_t}.
	\label{eq:Mundline}
	\end{equation}
	% \begin{remark} The sign of $\nabla_t$ determines the direction of movement of $x_t$ relative to $x_{t-1}$. Specifically, $x_t$ moves to the right [XXX: can you not say towards $x^\star$ or something more precise?] if $\nabla_t < 0$ and to the left if $\nabla_t > 0$. In this case we consider $\nabla_t < 0$. For simplicity, we take $\nabla_t = |\nabla_t|$ and omit the absolute value notation throughout the analysis.
	% \end{remark}
	\begin{lemma} \label{lemma:overline-UB}
	Lower bound of $\overline{\mathcal{H}_t}$ in terms of $\overline{\mathcal{M}_t}$ is given by the following
	\begin{equation}
		\overline{\mathcal{H}_t} \geq \left(\delta_t-\frac{\alpha L}{\nabla_t+\delta_t-\alpha}\right)
		\left(\frac{\nabla_t+\delta_t-\alpha}{\nabla_t+\delta_t}\right)\overline{\mathcal{M}_t}
	\end{equation}
	\end{lemma}
	\begin{proof}
	Using $L$-Lipschitzness of $f_t$ from  A.\ref{asm:a2} we get
	\begin{equation} \nonumber
		\begin{split}
			|f_t(x_{true})-f_t(\overline{x_t})| &\leq L|x_{true}-\overline{x_t}| \\ 
			\Rightarrow \mathcal{H}_t^{true} - \overline{\mathcal{H}_t} &\stackrel{(a)}\leq L\frac{\alpha \mathcal{M}_t^{true}}{\nabla_t+\delta_t-\alpha}\\
			\Rightarrow \overline{\mathcal{H}_t} &\geq \mathcal{H}_t^{true} - L\frac{\alpha \mathcal{M}_t^{true}}{\nabla_t+\delta_t-\alpha} \\
			&\stackrel{(b)}\geq \delta_t \mathcal{M}_t^{true} - \frac{\alpha L \mathcal{M}_t^{true}}{\nabla_t+\delta_t-\alpha} \\
			&\stackrel{\ref{eq:Mupline}}\geq \left(\delta_t-\frac{\alpha L}{\nabla_t+\delta_t-\alpha}\right)
			\left(\frac{\nabla_t+\delta_t-\alpha}{\nabla_t+\delta_t}\right)\overline{\mathcal{M}_t} 
		\end{split} 
	\end{equation}
	where (a) follows from \eqref{eq:Mupline} since $|x_t^{true}-\overline{x_t}|=\overline{\mathcal{M}_t}-\mathcal{M}_t^{true}$ and (b) follows from Lemma \ref{lemma:lb-cost}.
	\end{proof}
	
	\begin{lemma} \label{lemma:underline-LB}
	Lower bound of $\underline{\mathcal{H}_t}$ in terms of $\underline{\mathcal{M}_t}$ is given by the following
	\begin{equation*}
		\underline{\mathcal{H}_t} \leq \left(\delta_t+ \frac{M}{2}+\frac{\alpha L}{\nabla_t+\delta_t+\alpha}\right) \left(\frac{\nabla_t+\delta_t+\alpha}{\nabla_t+\delta_t}\right)\underline{\mathcal{M}_t}.
	\end{equation*}
	\end{lemma}
	\begin{proof}
	Using $L$-Lipschitzness of $f_t$ from A.\ref{asm:a2} we get $|f_t(x_{true})-f_t(\underline{x_t})| \leq L|x_{true}-\underline{x_t}|$
	\begin{equation} \nonumber
		\begin{gathered}
			\begin{aligned}
				\Rightarrow \underline{\mathcal{H}_t} - \mathcal{H}_t^{true} &\stackrel{(a)}\leq L\frac{\alpha \mathcal{M}_t^{true}}{\nabla_t+\delta_t+\alpha}\\
				\Rightarrow \underline{\mathcal{H}_t} &\leq \mathcal{H}_t^{true} + L\frac{\alpha \mathcal{M}_t^{true}}{\nabla_t+\delta_t+\alpha}\\
				&\stackrel{(b)}\leq \left(\delta_t + \frac{M}{2}\right)\mathcal{M}_t^{true} + \frac{\alpha L \mathcal{M}_t^{true}}{\nabla_t+\delta_t+\alpha}\\
				&\stackrel{(\ref{eq:Mundline})}\leq \left(\delta_t+ \frac{M}{2}+\frac{\alpha L}{\nabla_t+\delta_t+\alpha}\right) \left(\frac{\nabla_t+\delta_t+\alpha}{\nabla_t+\delta_t}\right)\underline{\mathcal{M}_t},
			\end{aligned}
		\end{gathered}
	\end{equation}
	where (a) follows from \eqref{eq:Mundline} because $|x_t^{true}-\overline{x_t}|=\mathcal{M}_t^{true}-\underline{\mathcal{M}_t}$ and (b) follows from Lemma \ref{lemma:ub-cost}.
	\end{proof}
	
	\begin{lemma}\label{lemma:received-LB-noisy}
	Under  A\ref{asm:a2} and A\ref{asm:a4}, $\mathcal{H}_t$ and $\mathcal{M}_t$ satisfies
	\begin{equation}\nonumber
		\begin{gathered}
			\begin{aligned}
				\mathcal{H}_t \geq \left(\delta_t - \frac{\alpha L}{\nabla_t+\delta_{t}-\alpha}\right)\left(\frac{\nabla_t+\delta_{t}-\alpha}{\nabla_t+\delta_{t}}\right)\mathcal{M}_t
			\end{aligned}
		\end{gathered}
	\end{equation}
	\end{lemma}
	\begin{proof} From inequalities \eqref{ineq:UB-LB-Mt} and \eqref{ineq:UB-LB-Ht}, we know that $\overline{\mathcal{M}_t} \geq \mathcal{M}_t$ and $\mathcal{H}_t \geq \overline{\mathcal{H}_t}$. Substituting these bounds into Lemma~\ref{lemma:overline-UB}, we obtain the desired result. 
	\end{proof}
	
	\begin{lemma}\label{lemma:received-UB-noisy}
	Under  A\ref{asm:a2} and A\ref{asm:a4}, $\mathcal{H}_t$ and $\mathcal{M}_t$ satisfies
	\begin{equation}\nonumber
		\begin{gathered}
			\begin{aligned}
				\mathcal{H}_t \leq \left(\delta_t + \frac{M}{2} + \frac{\alpha L}{\nabla_t+\delta_{t}+\alpha}\right)\left(\frac{\nabla_t+\delta_{t}+\alpha}{\nabla_t+\delta_{t}}\right)\mathcal{M}_t.
			\end{aligned}
		\end{gathered}
	\end{equation}
	%where, $\underline{\mathcal{H}_t}$ and $\underline{\mathcal{M}_t}$ are the hitting and movement cost incurred when $\underline{\nabla} = \nabla_t-\alpha$ is received by the algorithm.
	\end{lemma}
	\begin{proof} From  \eqref{ineq:UB-LB-Mt} and \eqref{ineq:UB-LB-Ht}, we know that $\underline{\mathcal{M}_t} \leq \mathcal{M}_t$ and $\mathcal{H}_t \leq \underline{\mathcal{H}_t}$. Substituting these bounds into Lemma~\ref{lemma:underline-LB}, we obtain the desired result. 
	\end{proof}
	
	In this noisy gradient setting, Lemmas \ref{lemma:received-LB-noisy} and \ref{lemma:received-UB-noisy} are analogues to Lemma \ref{lemma:lb-cost} and \ref{lemma:ub-cost} in the fresh gradient setting respectively. Using  Lemmas \ref{lemma:received-LB-noisy}, \ref{lemma:received-UB-noisy} we can derive our required $a,b$ as described above. However this choice of $a,b$ depends on $\nabla_{t}$ or the true gradient value whose quantity is unknown to the algorithm. We can remove this dependence by using the inequality $\nabla_{t}^{avl}-\alpha \leq \nabla_{t} \leq \nabla_{t}^{avl}+\alpha$ which is shown below. 
	\begin{lemma}\label{lemma:noisy-bounds}
	In the noisy fresh gradient setting, $\mathcal{H}_t$ and $\mathcal{M}_t$ satisfies : $a\mathcal{M}_t \leq \mathcal{H}_t \leq b\mathcal{M}_t$
	% [XXX: is this for underline or overline $M_t, H_t$????]
	where,
	\begin{equation}\nonumber
		\begin{gathered}
			\begin{aligned}
				a &=  \left(\delta_t - \frac{\alpha L}{\nabla_{t}^{avl}+\delta_{t}-2\alpha}\right)\left(1-\frac{\alpha}{\nabla_{t}^{avl}+\delta_{t}-\alpha}\right)\\
				b &= \left(\delta_t + \frac{M}{2} + \frac{\alpha L}{\nabla_{t}^{avl}+\delta_{t}}\right)\left(1+\frac{\alpha}{\nabla_{t}^{avl}+\delta_{t}-\alpha}\right)
			\end{aligned}
		\end{gathered}
	\end{equation}
	\end{lemma}
	
	% Note that $a,b$ are not constant terms, rather are dependent on $\delta_t$ and $\nabla_t^{avl}$. We observe that the $\nabla_{t}^{avl}$ in the denominator of $a,b$ always appears in additive relation with $\delta_{t}.$ If we choose $\nabla_{t}+\delta_{t}=\beta$ where $\beta$ is a constant, we can remove the dependence of $\nabla_t$ in the denominator. For the $\delta_{t}$ terms in numerator which can be written as $\beta-\nabla_{t}^{avl},$ we can use assumption~\ref{asm:a3}, \ref{asm:a4} to prove our Theorem~\ref{Theorem:UB-stale}.  
	% \hs{Will finish the below part afterwards.}
    \begin{proof}
    Combining Lemma~\ref{lemma:received-LB-noisy},\ref{lemma:received-UB-noisy} we obtain
	\begin{equation} 
    \left(\delta_t - \frac{\alpha L} {\nabla_t+\delta_t-\alpha}\right) \left(1-\frac{\alpha}{\nabla_t+\delta_t}\right)\mathcal{M}_t 
    \leq \mathcal{H}_t \nonumber \\ \leq 
    \left(\delta_t+ \frac{M}{2}+\frac{\alpha L}{\nabla_t +\delta_t+\alpha}\right) \left(1+\frac{\alpha}{\nabla_t+\delta_t}\right)\mathcal{M}_t \label{eq:noisy_grad}
\end{equation}
	Using $\nabla_t^{avl}-\alpha \leq \nabla_t \leq \nabla_t^{avl}+\alpha$ we can write \eqref{eq:noisy_grad} as
	\begin{equation}
				\left(\delta_t-\frac{\alpha L} {\nabla_t^{avl}+\delta_t-2\alpha}\right) \left(1 - \frac{\alpha}{\nabla_t^{avl}+\delta_t-\alpha}\right)\mathcal{M}_t \leq \mathcal{H}_t \nonumber \\ \leq \left(\delta_t+ \frac{M}{2}+\frac{\alpha L}{\nabla_t^{avl} +\delta_t}\right) \left(1+\frac{\alpha} {\nabla_t^{avl}+\delta_t-\alpha}\right)\mathcal{M}_t
				\label{eq:noisy_grad1}
	\end{equation}
    This completes the proof of Lemma~\ref{lemma:noisy-bounds}
    \end{proof}
	%	\hs{Write the above as a theorem}\\
\vspace{-0.15in}
    \begin{remark}[Choice of $\delta_t$]\label{remark:beta}
		We note that constants $a,b$ depend on both $\delta_t$ and $\nabla_t^{avl}$. In particular, the term $\nabla_t^{avl}$ appears additively with $\delta_t$ in the denominators of $a,b$. This is problematic as $\nabla_t^{avl}$ is controlled by the adversary, making any bounds depending on it non-robust. To eliminate this dependence, we choose $\delta_t$ such that $\nabla_{t}+\delta_{t}=\beta$ ($\beta$ is a constant), equivalently we set $\delta_t = G +L  - \nabla_t^{avl}$. Although this substitution induces $\delta_t$-dependent terms in the numerator of $a,b$. We can simplify these constants $a,b$ further by writing $\delta_t$ as $\beta - \nabla_t^{avl}$ and use assumptions A \ref{asm:a3} and A \ref{asm:a4} to obtain simplified bounds for $\mathcal{H}_t$.
	\end{remark}
	
    Combining the Lemma~\ref{lemma:noisy-bounds} with above mentioned choice of $\delta_t$ we obtain these new bounds  
	\begin{equation}
    \hat{a}\mathcal{M}_t \leq \mathcal{H}_t \leq \hat{b}\mathcal{M}_t \label{eq:new-ab} \\
		\hat{a} =	\left(L-\frac{\alpha L} {G+L-2\alpha}\right) \left(1-\frac{\alpha}{G+L-\alpha}\right) \nonumber\\
            \hat{b} =  \left(G+L+ \frac{M}{2}+\frac{\alpha L}{G+L}\right) \left(1+\frac{\alpha}{G+L-\alpha}\right) \nonumber
	\end{equation}

    Next, we use the same potential function method as in previous settings, with the key modification being the use of bounds from \eqref{eq:new-ab} that account for gradient noise. By employing this method, we establish that an upper bound on competitive ratio is given by
	$\mu_{\text{A-OBD}(\delta_t)} \leq \max\left\{1+\frac{2+\hat{b}}{\hat{a}}, 1+\hat{b}\right\},$
    where the constants $\hat{a}$ and $\hat{b}$ are defined as in \eqref{eq:new-ab}. This completes the proof of Theorem \ref{Theorem:UB-noisy}.
\vspace{-0.14in}
	\bibliographystyle{IEEEtran} %ACM-Reference-Format}
        \bibliography{IEEE-SingleGradient-OCO}
\end{document}